\documentclass[11pt]{amsart}

\usepackage{latexsym}
\usepackage{amsfonts}

\def\cal{\mathcal}

\newcommand{\field}[1]{\mathbb{#1}}
\newcommand{\C}{\field{C}}

\newcommand{\R}{\field{R}}

\newtheorem{defi}{Definition}[section]
\newtheorem{ex}[defi]{Example}
\newtheorem{lem}[defi]{Lemma}
\newtheorem{theo}[defi]{Theorem}
\newtheorem{co}[defi]{Corollary}
\newtheorem{pr}[defi]{Proposition}
\newtheorem{re}[defi]{Remark}

\font\tenmsy=msbm10

\def\Bbb#1{\hbox{\tenmsy#1}} 

\subjclass{14 D 99, 14 R 99, 51 M 99}

\title[On a generic  symmetry defect hypersurface]{On  a generic symmetry defect hypersurface} \makeatletter

\author{S. Janeczko, Z. Jelonek, M.A.S. Ruas}

\address[S. Janeczko ] {Instytut Matematyczny\\
Polska Akademia Nauk\\
\'Sniadeckich 8, 00-956 Warszawa, Poland\\
Wydzia\l  \ Matematki i Nauk Informacyjnych\\
Politechnika Warszawska, Pl. Politechniki 1, 00-661 Warszawa,
Poland } \email{janeczko@impan.pl}
\address[Z. Jelonek]{Instytut Matematyczny\\
Polska Akademia Nauk\\
\'Sniadeckich 8, 00-956 Warszawa, Poland}
\email{najelone@cyf-kr.edu.pl}
\address[M.A.S. Ruas]{Departamento de Matem\'atica,
ICMC-USP, Caixa Postal 668, 13560-970 S\~ao Carlos, S.P., Brasil}
\email{maasruas@icmc.usp.br}

\date{\today}

\begin{document}

\maketitle

\begin{abstract}{Let $f: X\to Y$ be a dominant polynomial mapping
of affine varieties.  For generic $y\in Y$ we have
$Sing(f^{-1}(y))=f^{-1}(y)\cap Sing(X).$ As an application we show
that symmetry defect hypersurfaces for two generic members  of the
irreducible algebraic family of $n-$dimensional smooth irreducible
subvarieties in general position in $\C^{2n}$ are homeomorphic and
they have homeomorphic sets of singular points. In particular
symmetry defect curves for two generic curves in $\C^2$  of the
same degree have the same number of singular points.}
\end{abstract}

\section{Introduction}
Let $X^n \subset \C^{2n}$ be a smooth algebraic variety. In
\cite{jjr} we have investigated the central symmetry  of  $X$ (see
also \cite{GH}, \cite{GJ}, \cite{Jan}).  For $p\in \C^{2n}$ we
have introduced a number $\mu (p) $ of pairs of points $x,y\in X$,
such that $p$ is {the} center of the interval $\overline{xy}$.
Recall that the subvariety $X^n\subset \C^{2n}$ is in a general
position if there exist points $x,y\in X^n$ such that $T_xX\oplus
T_yY=\C^{2n}.$

We have showed in \cite{jjr} that if $X$ is in general position,
then there is a closed algebraic hypersurface $B\subset \C^{2n},$
called {\it symmetry defect hypersurface} of  $X$, such that the
function $\mu$ is  constant (non-zero) exactly outside $B.$ Here
we prove that the symmetry defect hypersurfaces for two generic
members of an irreducible algebraic family  of $n-$dimensional
smooth irreducible subvarieties in general position in $\C^{2n}$
are homeomorphic.

Moreover, we prove a version of Sard theorem for singular
varieties (section 2), which implies  that additionally the
symmetry defect hypersurfaces for two generic members of an
irreducible algebraic family  of $n-$dimensional smooth
irreducible subvarieties in general position in $\C^{2n}$ have
homeomorphic sets of singular points. In particular symmetry
defect curves for two generic curves in $\C^2$  of the same degree
have the same number of singular points.

\section{Generalized Sard's Theorem}
Let $X$ be an irreducible affine variety. Let $Sing(X)$ denote the
set of singular points of $X.$ Let $Y$ be another affine variety
and consider a dominant morphism $f: X\to Y.$ If $X$ is smooth
then by Sard's Theorem a generic fiber of $f$ is smooth. We show
that in a general case the following theorem holds:

\begin{theo}\label{ber}
Let $f: X^k\to Y^l$ be a dominant polynomial mapping of affine
varieties.  For generic $y\in Y$ we have
$Sing(f^{-1}(y))=f^{-1}(y)\cap Sing(X).$
\end{theo}

\begin{proof}
We can assume that $Y$ is smooth. Since there exists a mapping
$\pi: Y^l\to \C^l$ which is generically etale, we can assume that
$Y=\C^l.$ Let us recall that if $Z$ is an algebraic variety, then
a point $z\in Z$ is smooth, if and only if the local ring ${\cal
O}_z(Z)$ is regular. This is equivalent to the fact  that $\dim_\C
{\mathfrak m}/{\mathfrak m}^2=\dim Z,$ where $\mathfrak m$ denotes
the maximal ideal of ${\cal O}_z(Z)$.

Let $y=(y_1,...,y_l)\in \C^l$ be a sufficiently generic point.
Then by Sard's Theorem the fiber $Z=f^{-1}(y)$ is smooth outside
$Sing(X)$ and $\dim Z=\dim X - l=k-l.$ Since the generic fiber of
$f$ is reduced (we are in characteristic zero!), then also a
generic particular fiber is reduced. Hence we can  assume that $Z$
is reduced. It is enough to show that every point $z\in Z\cap
Sing(X)$ is singular on $Z.$

Assume that $z\in Z\cap Sing(X)$ is smooth on $Z.$ Let $f: X\to
\C^l$ be given as $f=(f_1,...,f_l)$, where $f_i\in \C[X].$ Then
${\cal O}_z(Z)={\cal O}_z(X)/(f_1-y_1,...,f_l-y_l).$ In particular
if $\mathfrak m'$ denotes the maximal ideal of ${\cal O}_z(Z)$ and
$\mathfrak m$ denotes the maximal ideal of ${\cal O}_z(X)$ then
$\mathfrak m'=\linebreak \mathfrak m/(f_1-y_1,...,f_l-y_l).$ Let
$\alpha_i$ denote a class of the polynomial $f_i-y_i$ in
${\mathfrak m}/{\mathfrak m}^2.$ Let us note that
\begin{equation}\label{eq}
{\mathfrak m'}/{\mathfrak m'}^2={\mathfrak m}/({\mathfrak
m}^2+(\alpha_1,...,\alpha_l)).
\end{equation}
Since the point $z$ is smooth on $Z$ we have $\dim_\C {\mathfrak
m'}/{\mathfrak m'}^2=\dim Z=\dim X-l.$ Take a basis
$\beta_1,...,\beta_{k-l}$ of the space $ {\mathfrak m'}/{\mathfrak
m'}^2$ and let $\overline{\beta_i}\in {\mathfrak m}/{\mathfrak
m}^2$ correspond to $\beta_i$ under the correspondence (\ref{eq}).
Note that vectors $\overline{\beta_1},...,\overline{\beta_{k-l}},
\alpha_1,..., \alpha_l$ generates the space ${\mathfrak
m}/{\mathfrak m}^2.$ This means that $\dim_\C {\mathfrak
m}/{\mathfrak m}^2\le k-l+l=k=\dim X.$ Hence the point $z$ is
smooth on $X$, a contradiction.
\end{proof}

\begin{ex}
{\rm Note  that the general singularities of generic fibers of
$f: X\to Y$ are not necessarily simpler than those of $X$. Indeed,
if $X= Z\times \C^k$ and $f: X\to \C^k$ is a projection, then
the singularities of $X$ and the fiber $f^{-1}(y)$ are of the same
type.}
\end{ex}

\section{Bifurcation set}

Let $k=\C$ or $k=\R$  and let $X, Y$ be affine varieties over $k.$
Recall the following (see \cite{jel}, \cite{jel1}):

\begin{defi}
{\rm Let $f : X \rightarrow Y$ be a  generically-finite (i.e. a
generic fiber is finite) and dominant (i.e. $\overline{f(X)}=Y$)
polynomial mapping of affine varieties. We say that $f$ {\it is
finite at a point} $y \in Y,$ if  there exists an open
neighborhood $U$ of $y$ such that  the mapping $ f\mid_{f^{-1}(U)}
:f^{-1} (U)\rightarrow U$ is proper.}
\end{defi}

If $k=\C$ it is well-known that the set $S_f$ of points at which
the mapping $f$ is not finite, is either empty or it is a
hypersurface (see \cite{jel}, \cite{jel1}). We say that the set
$S_f$ is {\it the set of non-properness} of the mapping $f.$

\begin{defi}
Let $k=\C.$ Let $X,Y$ be  smooth affine $n-$dimensional varieties
and let $f : X\to Y$ be a generically finite dominant mapping of
geometric degree $\mu(f).$ The bifurcation set  of the mapping $f$
is the set
$$B(f)=\{y\in  Y : \# f^{-1}(y)\not=\mu(f)\}.$$
\end{defi}

We have the following fundamental theorem:

\begin{theo}\label{jk}
Let $k=\C.$ Let $X,Y$ be  smooth affine complex varieties of
dimension $n.$ Let $f: X \to Y$ be a polynomial dominant mapping.
Then the set $B(f)$ is either empty (so $f$ is an unramified
topological covering) or it is a closed hypersurface.
\end{theo}

\begin{proof}
Let us note that outside the set $S_f$ the mapping $f$ is a
(ramified) analytic cover of degree $\mu(f).$ By the Lemma \ref{l}
below if $y\not\in S_f$ we have $\# f^{-1}(y)\le \mu(f).$
Moreover, since $f$ is an analytic covering outside $S_f$ it is
well known that the fiber $f^{-1}(y)$ counted with multiplicity
has exactly $\mu(f)$ points. In particular, if $y\in K_0(f)$, {the
set of critical values of $f,$} then $\# f^{-1}(y)<\mu(f).$

Now let $y\in S_f.$ There are {two} possible  cases :

a)  $\# f^{-1}(y)=\infty.$

b)  $\# f^{-1}(y)<\infty.$

In  case b) let $U$ be an affine neighborhood of $y$ over which
the mapping $f$ is quasi-finite. Let $V=f^{-1}(U).$ By  Zariski
Main Theorem in the version given by Grothendieck, there exists a
normal variety $\overline{V}$ and a finite mapping
$\overline{f}:\overline{ V}\to U$, such that

1) $V\subset  \overline{V}$,

2) $\overline{f}\mid_V=f.$

Since  $y\in \overline{f}(\overline{V}\setminus V)$, it follows by
the Lemma \ref{l} below, that $\# f^{-1}(y)<\mu(f).$ Consequently,
if $y\in S_f,$ we have $\# f^{-1}(y)<\mu(f).$ Finally we have
$B(f)= K_0(f)\cup S_f.$

Now we show that the set $B(f)=K_0(f)\cup S_f$ is a hypersurface.
Let $J(f)$ be  the set of singular points of $f.$  The set $J(f)$
is a hypersurface (because locally it is the zero set of the
Jacobian of $f$). Denote by  $J_i$  {the} irreducible components
of {$J(f).$} Let $W_i=\overline{f(J_i)}.$ If all $W_i$ are
hypersurfaces then the theorem is true. If, for example dim $W_1<
n_1$, then the mapping $f : J_1\to W_1$ has non-compact generic
fiber, this means in particular that $W_1\subset S_f.$ Thus the
set $\bigcup W_i\cup S_f$ is a hypersurface. But $B(f)=\bigcup
W_i\cup S_f$ (note that $B(f)$ is closed).

Moreover, if $B(f)=\emptyset$,  then $f$ is a surjective
topological covering.
\end{proof}

\begin{lem}\label{l}
Let $X,Y$ be affine normal varieties of dimension $n.$ Let $f:X\to
Y$ be a finite mapping. Then for every $y\in Y$ we have $\#
f^{-1}(y)\le\mu(f).$
\end{lem}

\begin{proof}
Let $\# f^{-1}(y)=\{ x_1,..., x_r\}.$ We can choose a function
$h\in \C[X]$ which separates all $x_i$ (in particular we can take
as $h$ the equation of a general hyperplane section). Since $f$ is
finite we have a monic polynomial
$T^s+a_1(f)T^{{s-1}}+...+a_s(f)\in f^*\C[Y][T]$,  $ s\le \mu(f).$
If we substitute $f=y$ to this equation we get  {the} desired
result.
\end{proof}

\section{A super general position}

In this section we describe some properties of a variety
$X^n\subset \C^{2n}$ which implies that $X$ is in a general
position. Recall that the subvariety $X^n\subset \C^{2n}$ is in a
general position if there  exist  points $x,y\in X^n$ such that
$T_xX\oplus T_yY=\C^{2n}.$

\begin{defi}
Let $X^n\subset\C^{2n}$ be a smooth algebraic variety. We say that
$X$ is in very general position if there exists a point $x\in X$
such that the set $T_xX\cap X$ has an isolated point (here we
consider $T_x X$ as a linear subspace of $\C^{2n}$).
\end{defi}

We consider also a slightly stronger property:

\begin{defi}
Let $X^n\subset\C^{2n}$ be a smooth algebraic variety and let
$S=\overline{X}\setminus X\subset \pi_\infty$ be the set of points
at infinity of $X^n.$ We say that $X$ is in super general position
if there exists a point $x\in X$ such that  $T_xX\cap S=\emptyset$
(here we consider $T_x X$ as a linear subspace of $\Bbb
P^{2n}=\C^{2n}\cup \pi_\infty$).
\end{defi}

We have the following:

\begin{pr}
If $X$ is in a  super general position, then it is in a very
general position.
\end{pr}

\begin{proof}
Let $x\in X$ be a point such that $T_xX\cap S=\emptyset.$ Take
$R=T_xX\cap X.$ Then the set $R$ is finite, since otherwise the
point at infinity of $R$ belongs to $T_xX\cap S=\emptyset.$
\end{proof}

We have also:

\begin{pr}
Let $X\subset \C^{2n}$ be in a super general position. Then for a
generic point $x\in X$ we have $T_xX\cap S=\emptyset.$
\end{pr}

\begin{proof}
It is easy to see that the set $\Gamma=\{ (s,x)\in S\times X :
s\in T_xX \}$ is an algebraic subset of $S\times X.$ Let $\pi:
\Gamma\ni (s,x)\to x\in  X$ be a projection. It is a proper
mapping. Since the variety $X$ is in a very general position, we
see that at least one point $x_0\in X$ is not in the image of
$\pi.$ Thus almost every point of $X$ is not in the image of
$\pi$, because the image of $\pi$ is a closed subset of $X.$
\end{proof}

Finally we have:

\begin{theo}
If $X\subset\C^{2n}$ is in a very general position, then it is in
a general position, i. e., there exist points $x,y\in X$ such that
$T_xX\oplus T_yX=\C^{2n}.$ In fact for every generic pair
$(x,y)\in X\times X$ we have $T_xX\oplus T_yX=\C^{2n}.$
\end{theo}

\begin{proof}
Let $x_0\in X$ be the point such that the set $T_{x_0}X\cap X$ has
an isolated point. The space $T_{x_0}X$ is given by $n$ linear
equations $l_i=0.$ Let $F: X\ni x \to (l_1(x),...,l_n(x))\in
\C^n.$ By the assumption the fiber over $0$ of $F$ has an isolated
point, in particular the mapping $F$ is dominant. Now by the Sard
Theorem almost every point $x\in X$ is a regular point of $F.$
This means that $T_xX$ is complementary to $T_{x_0}X$, i.e.,
$T_{x_0}X\oplus T_xX=\C^{2n}.$ If we consider the mapping $\Phi:
X\times X\ni (x,y)\to x+y\in \C^{2n}$, we see that it has the
smooth point $(x_0,x).$ In particular almost every pair $(x,y)$ is
a smooth point of $F,$ which implies that for every generic pair
$(x,y)\in X\times X$ we have $T_xX\oplus T_yX=\C^{2n}.$
\end{proof}

We shall use in the sequel the following:

\begin{pr}\label{general}
Let $X^n\subset \C^{2n}$ be a generic smooth complete intersection
of multi-degree $d_1,..., d_n.$ If every $d_i>1$, then $X$ is in a
super general position.
\end{pr}

\begin{proof}
We can assume that $X$ is given by $n$ smooth hypersurfaces $f_i=
a_i+ f_{i1}+...+f_{id_i}$ (where $f_{ik}$ is a homogenous
polynomial of degree $k$), which have independent all coefficients
(see section below). The tangent space is described by polynomials
$f_{i1}, i=1,...,n$ and the set $S$ of points at infinity of $X$
is described by polynomials $f_{id_i}, i=1,...,n.$ Since these two
families of polynomials have independent coefficients, we see that
generically the zero sets at infinity of these two families are
disjoint. In particular such a generic $X$ is in a super general
position.
\end{proof}

\section{Algebraic families}

Now we introduce the notion of an algebraic family.

\begin{defi}
Let $M$ be a smooth affine algebraic variety and let $Z$ be  a smooth irreducible subvariety of
$M\times \C^n$ . If the restriction to $Z$ of the  projection
$\pi: M\times \C^n\to M$ is a dominant map with generically irreducible
fibers of the same  dimension, then we call  the collection $\Sigma=\{
Z_m=\pi^{-1}(m)\}_{m\in M}$  an algebraic family of
subvarieties in $\C^n.$ We say that this family is in a general
position if a generic member of $\Sigma$ is in a general position
in $\C^n.$
\end{defi}

We show that the ideals $I(Z_m)\subset \C[x_1,..., x_n]$ of a generic
member of $\Sigma$ depend in a parametric way on $m\in M.$

\begin{lem}\label{ideal}
Let $\Sigma$ be an algebraic family given by a smooth variety
$Z\subset M\times \C^n.$  The ideal $I(Z)\subset
\C[M][x_1,....,x_{n}]$ is finitely generated,  let the polynomials
$\{f_1(m,x),..., f_s(m,x)\}$ form its set of generators. The ideal
$I(Z_m)\subset \C[x_1,..., x_{n}]$ of a generic member
$Z_m:=\pi^{-1}(m)\subset \C^{n}$ of $\Sigma$ is equal to
$I(Z_m)=(f_1(m,x),..., f_s(m,x)).$
\end{lem}

\begin{proof}
Let dim $Z=p$ and dim $M=q.$ Thus the variety $M\times \C^n$ has
dimension $n+q.$ Choose  local holomorphic coordinates on $M.$
Since the variety $Z$ is smooth we have

\vspace{5mm}
\begin{center}

      ${\rm rank} \left[ \begin{array}{ccccccccc}

            \frac{\partial f_1}{\partial m_1}(m,x) & \ldots & \frac{\partial f_1}{\partial m_q}(m,x) & \frac{\partial f_1}{\partial x_1}(m,x) & \ldots & \frac{\partial f_{1}}{\partial x_n}(m,x) \\
            \vdots & & \vdots & \vdots & & \vdots &   \\
             \frac{\partial f_s}{\partial m_1}(m,x) & \ldots &  \frac{\partial f_s}{\partial m_q}(m,x) & \frac{\partial f_s}{\partial x_1}(m,x) & \ldots & \frac{\partial f_{s}}{\partial x_n}(m,x) \\
            \end{array}
   \right]=n+q-p $
\end{center}

\noindent on $Z.$ Let us consider the projection $\pi: Z\ni (m,x)
\mapsto m\in M.$ By Sard's theorem a generic  $m\in M$ is a
regular value of the mapping $\pi.$ For such a regular value $m$
we have that ker $d_{(m,x)}\pi$ is disjoint from $T_{(m,x)} Z$ for
every $x$ such that $(m,x)\in Z.$ In local coordinates on $M$ this
is equivalent to

\vspace{5mm}
\begin{center}

      ${\rm rank} \left[ \begin{array}{ccccccccc}
            1 & \ldots & 0 & 0 & \ldots & 0 \\
            \vdots & & \vdots & \vdots & & \vdots &  \\
            0 & \ldots & 1 & 0 & \ldots & 0 \\

            * & \ldots & * & \frac{\partial f_1}{\partial x_1}(m,x) & \ldots & \frac{\partial f_{1}}{\partial x_n}(m,x) \\
            \vdots & & \vdots & \vdots & & \vdots &   \\
            * & \ldots & * & \frac{\partial f_s}{\partial x_1}(m,x) & \ldots & \frac{\partial f_{s}}{\partial x_n}(m,x) \\
            \end{array}
   \right]=n+2q-p. $
\end{center}

\noindent Consequently for $(m,x)\in Z$ and $m$ a regular value of
$\pi$ we have

\vspace{5mm}
\begin{center}

 ${\rm rank} \left[ \begin{array}{ccccccccc}

             \frac{\partial f_1}{\partial x_1}(m,x) & \ldots & \frac{\partial f_{1}}{\partial x_n}(m,x) \\
           \vdots & & \vdots &   \\
            \frac{\partial f_s}{\partial x_l}(m,x) & \ldots & \frac{\partial f_{s}}{\partial x_l}(m,x) \\
            \end{array}
   \right]=n+q-p. $
\end{center}

\noindent  Note that $n+q-p={\rm codim}\ Z_m$ ( in $\C^n$). This
means that the ideal $(f_1(m,x),..., f_s(m,x))$ locally coincide
with $I(Z_m),$ because it contains local equations of $Z_m.$ Hence
it also coincides globally, i.e., $(f_1(m,x),...,
f_s(m,x))=I(Z_m).$
\end{proof}

\begin{re}
{\rm This can be also obtained by a computation of a scheme
theoretic fibers of $\pi$ and using the fact that such generic
fibers are reduced.}
\end{re}

\begin{ex}{\rm
a) Let $N:={n+d \choose d}$ and let $Z\subset \C^N\times \C^n$  be
given by equations $Z=\{ (a, x)\in \C^N\times \C^n:
\sum_{|\alpha|\le d} a_\alpha x^\alpha=0\}.$ The projection $\pi
:Z\ni (a, x)\to a\in \C^N$ determines an algebraic family of
hypersurfaces of degree $d$ in $\C^n.$ If $n=2$ and $d>1$ this
family is  in  general position in $\C^2.$

b) More generally let $N_1:={n+d_1 \choose d_1}$, $N_2:={n+d_2
\choose d_2}$, $N_n:={n+d_n \choose d_n}$ and let $Z\subset
\C^{N_1}\times \C^{N_2}...\times \C^{N_n}\times \C^{2n}$ be given
by equations $Z=\{ (a_1,a_2,...,a_n, x)\in \C^{N_1}\times
\C^{N_2}...\times \C^{N_n}\times \C^n: \sum_{|\alpha|\le d_1}
{a_1}_\alpha x^\alpha=0, \sum_{|\alpha|\le d_2} {a_2}_\alpha
x^\alpha=0 ,..., \sum_{|\alpha|\le d_n} {a_n}_\alpha
x^\alpha=0\}.$ The projection $\pi :Z\ni (a_1, a_2,...,a_n, x)\to
(a_1, a_2,..., a_n)\in \C^{N_1}\times \C^{N_2}...\times \C^{N_n}$
determines an algebraic family $\Sigma(d_1,d_2,..., d_n, 2n)$ of
complete intersections of multi-degree $d_1, d_2,...., d_n$ in
$\C^{2n}.$ If $ d_1, d_2,..., d_n>1 $, then  this family is in
general position in $\C^{2n}.$ This follows from Proposition
\ref{general}. }
\end{ex}

\section{Main result}
Let us recall that a following result is true ( see e.g.
\cite{jjr}):

\begin{lem}\label{fiber}
Let $X, Y$ be  complex algebraic varieties and $f:X\to Y$ a
polynomial dominant mapping. Then two generic fibers of $f$ are
homeomorphic.
\end{lem}

\begin{proof}
Let $X_1$ be an algebraic completion of $X.$ Take
$X_2=\overline{graph(f)}\subset X_1\times \overline{Y},$ where
$\overline{Y}$ is a smooth algebraic completion of $Y.$ We can
assume that $X\subset X_2.$ Let $Z=X_2\setminus X.$ We have an
induced mapping $\overline{f}: X_2\to  \overline{Y},$ such that
$\overline{f}_X=f.$

There is a Whitney stratification $\cal S$  of the pair $(X_2,Z).$
For every smooth strata $S_i\in \cal S$ let $B_i$ be the set of
critical values of the mapping {${f}|_{S_i}.$}Take $B=\overline{
\bigcup B_i}.$ Take $X_3=X_2\setminus f^{-1}(B)$ and
$Z_1=Z\setminus f^{-1}(B).$ The restriction of the stratification
$\cal S$ to $X_3$ gives a Whitney {stratification}  of the pair
$(X_3, Z_1).$ We have a proper mapping $f_1 : X_3\to
\overline{Y}\setminus B$ which is submersion on each strata. By
the Thom first isotopy theorem there is a trivialization of $f_1$,
which preserves the strata. It is an easy observation that this
trivialization gives a trivialization of the mapping $f:
X\setminus f^{-1}(B)\to Y\setminus B.$
\end{proof}

\begin{defi}
{\rm Let $X$ be an affine variety. Let us define
$Sing^k(X):=Sing(X)$ for $k:=1$ and inductively
$Sing^{k+1}(X):=Sing(Sing^k(X)).$}
\end{defi}

As a direct application of the Lemma \ref{fiber} and Theorem
\ref{ber} we have:

\begin{theo}\label{w}
Let $f: X^n\to Y^l$ be a dominant polynomial mapping of affine
varieties.  If $y_1, y_2$ are sufficiently general then
$f^{-1}(y_1)$ is homeomorphic to $f^{-1}(y_2)$ and
$Sing(f^{-1}(y_1))$ is homeomorphic to $Sing(f^{-1}(y_2)).$ More
generally, for every $k$ we have $Sing^k(f^{-1}(y_1))$ is
homeomorphic to $Sing^k(f^{-1}(y_2)).$
\end{theo}

Now we are ready to prove:

\begin{theo}
Let $\Sigma$ be an algebraic family of $n-$dimensional algebraic
subvarieties in $\C^{2n}$ in  general position.  Symmetry defect
hypersurfaces $B_1, B_2$ for generic members $C_1, C_2\in \Sigma$
are homeomorphic and they have homeomorphic singular parts i.e.,
$Sing(B_1)\cong Sing(B_2).$ More generally, for every $k$ we have
$Sing^k(B_1)$ is homeomorphic to $Sing^k(B_2).$
\end{theo}

\begin{proof}
Let $\Sigma$ be given by a variety $Z\subset M\times \C^{2n}.$ The
ideal $I(Z)\subset \C[M][x_1,....,x_{2n}]$ is finitely generated.
Choose a finite set of generators $\{f_1(m,x),..., f_s(m,x)\}.$

By Sard Theorem we can assume that all fibers of $\pi: Z\to M$ are
smooth and for every $m\in M$ we have $I(Z_m)=\{f_1(m,x),...,
f_s(m,x)\}$ (see Lemma \ref{ideal}). Let us define
$$
R=\{ (m,x,y) \in M\times \C^{2n}\times \C^{2n}: f_i (m)(x)=0,
i=1,..., s \quad \& \quad f_i (m)(y)=0, i=1,...,s\}.
$$

The variety $R$ is a smooth irreducible subvariety of $ M\times
\C^{2n}\times \C^{2n}$ of codimension $2n$. Indeed, for given
$(m,x,y)\in M\times \C^{2n}\times \C^{2n}$ choose polynomials
$f_{i_1},...,f_{i_n}$ and $f_{j_1},...,f_{j_n}$ such that rank
$[\frac{\partial f_{i_l}}{\partial x_s}(m,x)]_{l=1,...,n;
s=1,...,n}=n$ and rank $[\frac{\partial f_{j_l}}{\partial
x_s}(m,x)]_{l=1,...,n; s=1,...,n}=n.$ Since $Z$ is a smooth
variety of dimension $\dim M + n$, we have that $Z$ locally near
$(m,x)$ is given by equations $f_{i_1},...,f_{i_n}$ and near
$(m,y)$ is given by equations $f_{j_1},...,f_{j_n}.$ Hence the
variety $R$ near the point $(m,x,y)$ is given as $$\{ (m,x,y) \in
M\times \C^{2n}\times \C^{2n}: f_{i_l} (m)(x)=0, l=1,..., n \quad
\& \quad f_{j_l} (m)(y)=0, l=1,...,s\}.$$ In particular $R$ is
locally a smooth complete intersection, i.e., $R$ is smooth.

Moreover we have a projection $R\to M$ with irreducible fibers
 which are products $Z_m\times Z_m, \quad m\in M$. This
means that $R$ is irreducible. Note that $R$ is an affine variety.
Consider the following  morphism
$$
\Psi :R\ni (m,x,y)\mapsto (m, {x+y\over 2})\in M\times \C ^{2n}.
$$  By the assumptions the mapping $\Psi$ is dominant. Indeed for every
$m \in M$  the fiber $Z_m$ is in a general position in $\C^{2n}$
and consequently the set $\Psi(R) \cap m\times \mathbb C^{2n}$ is
dense in $m\times \mathbb C^{2n}.$

We know by Theorem \ref{jk} that the mapping $\Psi$ has constant
number of points in the fiber outside the bifurcation set $B(\Psi
)\subset M\times \C ^{2n}.$ This implies that $B(Z_m)=m\times
\C^{2n}\cap B(\Psi ).$ In particular the symmetry defect
hypersurface of the variety $Z_m$ coincide with the fiber over $m$
of the projection $\pi :B(\Psi )\ni (m,x) \mapsto m\in M.$ Now we
conclude the proof by Theorem \ref{w}.
\end{proof}

\begin{co}
Symmetry defect sets $B_1, B_2$  for generic curves $C_1,
C_2\subset \C^2$ of the same degree $d {>1}$ are homeomorphic and
they have the same number of singular points.
\end{co}

\begin{co}
Let $C_1,C_2$ be two  smooth varieties, which are generic complete
intersection of multi-degree $d_1,d_2,..., d_n$ in $\C^{2n}$
(where all $d_i>1$). Then symmetry defect hypersurfaces $B_1, B_2$
of $C_1, C_2$, are homeomorphic and they have homeomorphic
singular parts (i.e., $Sing(B_1)\cong Sing(B_2)$). More generally,
for every $k$ we have $Sing^k(B_1)$ is homeomorphic to
$Sing^k(B_2).$
\end{co}

\end{document}